\documentclass{birkjour}
 \newtheorem{thm}{Theorem}[section]
 
 \newtheorem{lem}[thm]{Lemma}
 \newtheorem{prop}[thm]{Proposition}
 \theoremstyle{definition}
 \newtheorem{defn}[thm]{Definition}
 \theoremstyle{remark}
 \newtheorem{rem}[thm]{Remark}
 \newtheorem{ex}[thm]{Example}
 \numberwithin{equation}{section}

\usepackage{amsmath,amsthm,amsfonts,amscd,eucal}
\usepackage{url}
\usepackage{amssymb}
\usepackage{xcolor}
\usepackage{setspace}
\usepackage{hyperref}
\usepackage{dsfont}
\usepackage{tikz-cd}
\usepackage{microtype}

\def\ca{{\mathcal A}}
\def\cb{{\mathcal B}}

\def\ch{{\mathcal H}}

\def\ck{{\mathcal K}}

\def\car{{\mathcal R}}
\def\cs{{\mathcal S}}

\def\cv{{\mathcal V}}

\def\g{\gamma}  \def\G{\Gamma}
  
\def\eeps{\epsilon}
\def\eps{\varepsilon}

\def\m{\mu}

\def\r{\rho}
\def\s{\sigma} 

\def\f{\varphi}  
\def\th{\theta}  
\def\om{\omega}

\def\ga{{\mathfrak A}} 
\def\gb{{\mathfrak B}}

\def\bc{\ensuremath{\mathbb C}}

\def\bz{\ensuremath{\mathbb Z}}

\def\id{\hbox{id}}
\def\ker{\hbox{Ker}}

\def\min{\mathop{\rm min}}

\def\aut{\mathop{\rm Aut}}

\def\ran{\mathop{\rm Ran}}

\def\ad{\mathop{\rm ad}}

\def\idd{{1}\!\!{\rm I}}

\begin{document}


%
%
%
%
%
%
%
%
%

\title[]{$C^*$-independence for $\bz_2$-graded $C^*$-algebras}

\author[M. E. Griseta]{Maria Elena Griseta}

\address{Dipartimento di Matematica\\
Universit\`{a} degli Studi di Bari\\
Via E. Orabona, 4, 70125 Bari, Italy}

\email{mariaelena.griseta@uniba.it}

\author[P. Zurlo]{Paola Zurlo}

\address{Dipartimento di Matematica\\
Universit\`{a} degli Studi di Bari\\
Via E. Orabona, 4, 70125 Bari, Italy}

\email{paola.zurlo@uniba.it}



\date{April 6, 2025}

\begin{abstract}
We analyze a notion of $C^*$-independence for $\bz_2$-graded $C^*$-algebras. We provide other notions of statistical independence for $\bz_2$-graded von Neumann algebras and prove some relationships between them. We provide a characterization for the graded nuclearity property.
\end{abstract}

\maketitle

\section{Introduction}
\label{sec:intro}
The notion of independence in the noncommutative setting has been widely studied in recent decades due to its connection with noncommutative central limit theorems, Brownian motions and the law of small numbers. Some examples in this direction are given by free independence \cite{VoDyNi, Sp98}, Boolean independence \cite{SpW} and monotonic independence \cite{Mur01}.\\
\noindent Another notion of noncommutative independence, the $C^*$-independence or statistical independence, was first introduced in \cite{HaKa}, for its relevance in quantum field theory.  
Two $C^*$-subalgebras $\ga_1$ and $\ga_2$ of a given $C^*$-algebra $\ga$ are said statistical independent when any two marginal states on $\ga_1$ and $\ga_2$ respectively, admit a common extension on $\ga$. \\
In addition, the Schlieder condition, sometimes called $(S)$-independence, \emph{i.e.} $xy\neq0$ for given non-vanishing elements $x\in\ga_1$ and $y\in \ga_2$ of $C^*$-algebras $\ga_1$ and $\ga_2$, is a necessary condition for statistical independence. This result, due to Schlieder \cite{Sch} for the algebra of observables associated with a region of the Minkowski space, has been proved in \cite[Theorem 2.5]{GoLuWi} and \cite[Proposition 2.3.]{Ham97} for the general case. In addition, Roos \cite{Roos} showed that the $(S)$-independence is also a sufficient condition for $C^*$-independence if the subalgebras $\ga_1$ and $\ga_2$ commute elementwise. \\
\noindent Other notions of independence are given for von Neumann algebras in \cite{GoLuWi,Ham97}. More in detail, $C^*$-independence has a counterpart in von Neumann algebras by the so-called $W^*$-independence, which requires the natural condition of normality for the marginal states and their common extension.  It is also known that $W^*$-independence is stronger than logical independence first introduced by R\'{e}dei \cite{Red}, and this latter implies $C^*$-independence, \cite[Theorem 3.8]{GoLuWi}. 
Moreover, in \cite{Ham97} the author extended the result obtained by Roos to Jordan-Banach algebras. \\
This paper aims to analyze a notion of $C^*$-independence for $\bz_2$-graded $C^*$-algebras and its relationship with $W^*$-independence for $\bz_2$-graded von Neumann algebras. An example of $\bz_2$-graded $C^*$-algebra comes from the Canonical Anticommutation Relations algebra (CAR for short) (see \cite{CDF,CRZ2}), which provides applications to quantum field theory, statistical mechanics and quantum probability.\\
In the $\bz_2$-graded setting, we express the notion of $C^*$-independence and $W^*$-independence by means of the existence of a simultaneous even extension of two given even states on $C^*$-subalgebras or even and normal states for von Neumann subalgebras. $\bz_2$-graded $C^*$-algebras, also called superalgebras by physicists \cite{Lang}, are obtained by assigning an involutive $*$-automorphism on a $C^*$-algebra. In \cite{CRZ}, the authors proved that the spatial norm is the smallest between all compatible norms and that even states, \emph{i.e.} grading invariant states,  separate the elements of the $C^*$-algebras. Moreover, the product state on the involutive ($\bz_2$-graded) tensor product of two $\bz_2$-graded $C^*$-algebras is well defined if one of the two states is even \cite{CDF}. \\
In this paper, after giving a construction of the $\bz_2$-graded universal representation in Section \ref{sec:Z2graded}, we prove that $(S)$-independence is a necessary condition for $C^*$-independence in the $\bz_2$-graded setting. Moreover, we show that $(S)$-independence implies the aforementioned statistical independence if the algebras commute with the grading, \emph{i.e.} two odd elements anticommute (see Section \ref{sec:C*independence}).  As for $W^*$-independence for $\bz_2$-graded von Neumann algebras, we prove it is stronger than $C^*$-independence by passing, as in the case of trivial grading, through strict locality, cross property and logical independence (see Section \ref{sec:W*independence}). \\
In passing, we also take this opportunity to conclude the paper analyzing, in Section \ref{sec:nuclearity}, the nuclearity for $\bz_2$-graded $C^*$-algebras. Recall that it is possible to define several $C^*$-norms on the algebraic tensor product of two given $C^*$-algebras. Among those minimal and maximal play a privilege role \cite{Ka2}. A $\bz_2$-graded $C^*$-algebra $\ga$ is said to be $\bz_2$-nuclear if
the maximal and minimal $C^*$-cross norms on the $\bz_2$-graded tensor $\ga\hat{\otimes}\gb$ product are the same for every $\bz_2$-graded $C^*$-algebra $\gb$. Here, when the normal norm is compatible, in a sense that will be clarified in Section \ref{sec:nuclearity}, we characterize $\bz_2$-nuclearity by exploiting the normal tensor product of von Neumann algebras.\\
 
\section{$\bz_2$-graded $*$-algebras}
\label{sec:Z2graded}
In this section we start by recalling definitions and notions concerning $\bz_2$-graded $C^*$-algebras and Klein transformation. Successively, we provide the construction for $\bz_2$-graded enveloping von Neumann algebra and we recall some notions about the $\bz_2$-graded tensor products of $\bz_2$-graded $C^*$-algebras and their norms. We refer the reader to \cite{CRZ} and \cite{CDF} for further details. \\
Here and subsequently, all the considered structures will be taken as unitary. \\
Let $\bz_2=\{-1,1\}$ with the product as the group operation. A $*$-algebra $\ga$ is called an \emph{involutive $\bz_2$-graded algebra} if $\ga = \ga_1 \oplus \ga_{-1}$
and
\begin{align*}
  (\ga_i)^* &=(\ga^*)_i\,, & \ga_i\ga_j &\subset\ga_{ij}\,, & i,j &=1,-1\,.
\end{align*}
The subspaces $\ga_i$, $i=1,-1$ are called the \emph{homogeneous components} of $\ga$, and correspondingly any element of $\ga_i$ is called a homogeneous element of $\ga$. For any homogeneous element $x\in\ga_{\pm 1}$ we denote its {\it grade} by $\partial(x)=\pm1$. Assigning a $\bz_2$-grading on $\ga$ is equivalent to equipping $\ga$ with an involutive $*$-automorphism $\th$ (\emph{i.e.} $\th^2=\id_\ga$). Indeed, from one hand for a given $\bz_2$-graded $*$-algebra $\ga$ one takes $\th\lceil_{\ga_1} =\id_{\ga_1}$  and $\th\lceil_{\ga_{-1}} =-\id_{\ga_{-1}}$. 
On the other hand, if $\th\in\aut(\ga)$ is such that $\th^2=\id_\ga$, after taking $\eps_1 :=\frac{1}{2}(\id_{\ga}+\th)$ and $\eps_{-1} :=\frac{1}{2}(\id_{\ga}-\th)$ and denoting $\ga_1 :=\eps_1(\ga)$ and $\ga_{-1} :=\eps_{-1}(\ga)$, 
one gives $\ga_1\cap \ga_{-1}=\{0\}$. Consequently, their direct sum $\ga=\ga_1\oplus \ga_2$ is a $\bz_2$-graded $*$-algebra. Therefore, a $\bz_2$-graded $*$-algebra is a pair $(\ga,\th)$, where $\ga$ is an involutive $*$-algebra, and $\th$ an involutive $*$-automorphism on $\ga$. Following \cite{CDF}, we say that $\theta$ is a $\bz_2$-grading of $\ga$. Moreover, we denote the $*$-subalgebra $\ga_+:=\ga_1$ the {\it even part}, and the subspace $\ga_-:=\ga_{-1}$ the {\it odd part} of $\ga$, respectively. Note that $\eps_1$ is a conditional expectation, \emph{i.e.} a positive $\ga_{+}$-module projection from $\ga$ to $\ga_{+}$ (see \cite{CRZ}). Thus, for any $a\in\ga$, we can write $a=a_+  +a_-$, with $a_+\in\ga_+$, $a_-\in\ga_-$,
and this decomposition is unique. In addition, one gets $\th(a_+)=a_+$, $\th(a_-)=-a_-$. Taking $\th=\id_{\ga}$, one sees that any $*$-algebra $\ga$ is equipped with a $\bz_2$ trivial grading. Here, $\ga_+=\ga$ and $\ga_-=\{0\}$. \\  
A $\bz_2$-graded Hilbert space is a pair $(\ch,\G)$, where $\ch$ is a (complex) Hilbert space and $\G$ a self-adjoint unitary acting on $\ch$. Note that $\ch$ decomposes into a direct sum
$\ch=\ch_{+}\oplus\ch_{-}$, where $\ch_{+}:=\ker(I-\G)$, $\ch_{-}:=\ker(I+\G)$, and $I$ is the identity operator. Vectors belonging to $\ch_{+}$ ($\ch_{-}$) are referred to as \emph{even} (\emph{odd}) vectors, and elements belonging to any of these subspaces are collectively referred to as homogeneous vectors. The grade $\partial(\xi)$ of any homogeneous vector $\xi$ is $1$ or $-1$, according
to whether it belongs to $\ch_+$ or $\ch_-$, respectively.
\begin{defn}
\label{def:commgrad}
$\bz_2$-graded $C^*$-algebras are said to \emph{commute with grading}, $(C)_{\bz_2}$, if for each $a_1\in\ga_1$ and $a_2\in\ga_2$, one has $a_1a_2=\eps(a_1,a_2)a_2a_1$, where 
\begin{equation*}
\eps(a_1,a_2):=\begin{cases}
         -1, & \mbox{if} \, \partial(a_1)=\partial(a_2)=-1 \\
         1, & \mbox{otherwise}.
       \end{cases}
\end{equation*}
\end{defn}
Let $(\ga_i,\th_i)$, $i=1,2$, $\bz_2$-graded $*$-algebras. The map $T:\ga_1\rightarrow \ga_2$ is said to be {\it even} if $T\circ\th_1=\th_2\circ T$. When $(\ga_2,\th_2)=(\bc,\id_{\bc})$, a functional $f:\ga_1\rightarrow \bc$ is even if and only if $f\circ\th_1=f$. In what follows, we will denote the convex subset of all even states by $\cs_+(\ga)$. \\
Suppose that $(\ga,\th)$ is a $\bz_2$-graded $C^*$-algebra and $\f\in\cs_+(\ga)$. Let $(\ch_\f,\pi_\f, \xi_\f, \G_{\th,\f})$ be the GNS covariant representation of $\f$, where $\pi_{\f}: \ga \to \cb(\ch_{\pi})$ is a $*$-representation on $\ga$, with $\ch_{\pi}$ the corresponding Hilbert space, and the unitary self-adjoint $\G_{\th,\f}$ is such that $\Gamma_{\theta, \varphi}\pi_{\varphi}(a)\xi_\f=\pi_\f(\theta(a))\xi_\f$. Consequently, $\pi_{\f}$ is grading-equivariant, \emph{i.e.} verifies $\pi_\f(\th(a))=\G_{\th,\f}\pi_\f(a)\G_{\th,\f}$, $a\in\ga$. Then, putting $\g:=\ad_{\G_{\th,\f}}$ the adjoint action, $\ad_{\G}(\cdot):=\G\cdot \G^*$, $(\cb(\ch),\g)$ is a $\bz_2$-graded $C^*$-algebra and $\G$ is even.\\
A non-degenerate representation on $C^*$-algebras with trivial grading is the direct sum of a family of cyclic subrepresentations (see \cite[Proposition 2.3.6.]{BR1}). The following proposition shows that any grading equivariant representation $\pi_\f$ is unitarily equivalent to the direct sum of cyclic representations coming from even states. 
\begin{prop}
\label{prop:cyclic}
Let $\pi:\ga \to \cb(\ch_{\pi})$ be a grading-equivariant representation and denote by $S$ a proper subspace of $\cs_{+}(\ga)$. Then
$$
\pi=\bigoplus_{\om \in S \subset \cs_{+}(\ga)}\pi_{\om}
$$
up to unitary equivalence.
\end{prop}
\begin{proof}
One can suppose, without loss of generality, that $\G\neq \id_{\ch_{\pi}}$ (otherwise, there would be nothing to prove).\\
Let $\xi \in \ch_{\pi}$ be even. Note that the cyclic subspace generated by $\xi$   $$M:=\overline{\pi(\ga)\xi}$$  
is invariant by definition. Moreover, the vector state $\om$ associated with $\xi$, namely
$
\om(a):=\langle \pi(a)\xi,\xi\rangle
$, is even. Indeed, one has
\begin{align*}
\om(\th(a))&=\langle \pi(\th(a))\xi,\xi\rangle=\langle \G\pi (a)\G\xi,\xi\rangle=\langle \pi(a)\G\xi,\G\xi\rangle\\&=\langle \pi(a)\xi,\xi\rangle=\om(a)\, .
\end{align*}
By uniqueness (up to unitary equivalence) of the GNS representation, we have
$\pi_{\upharpoonright M}\cong\pi_{\om}$.\\
If $M=\ch_{\pi}$, there is nothing to prove.
Suppose $M$ is a proper subspace of $\ch_{\pi}$. Then  $M^{\perp}$ is different from $\{0\}$.
If $M^\perp \cap \ch_{\pi,+} \neq \{0\}$, then another cyclic subspace can be obtained in the same way as above.\\
If $M^{\perp} \cap \ch_{\pi,+} =\{0\}$, then we take $\xi$ in $M^\perp\cap \ch_{\pi,-} $.
Now the corresponding vector state is still even, for
\begin{align*}
\om(\th(a))&=\langle \pi(\th(a))\xi,\xi\rangle=\langle \G\pi (a)\G\xi,\xi\rangle=\langle \pi(a)\G\xi,\G\xi\rangle\\
&=\langle \pi(a)(-\xi),(-\xi)\rangle=\langle \pi(a)\xi,\xi\rangle=\om(a)\,.
\end{align*}
Finally, a standard application of Zorn's Lemma yields the sought decomposition.
\end{proof}
\begin{prop}
\label{prop:enveloping}
The unitary $V: \bigoplus_{\om \in S \subset \cs_{+}(\ga)} \ch_{\om} \to \ch_{\pi}$ which realizes the equivalence between $\pi$ and $\bigoplus_{\om \in S \subset \cs_{+}(\ga)}\pi_{\om}$ can be chosen to be grading-equivariant.
\end{prop}
\begin{proof}
Let us denote by $\G$ and $\bigoplus_{\om \in S \subset \cs_{+}(\ga)} \G_{\om}$ the unitaries (gradings) on the Hilbert spaces $\ch_\pi$ and $\bigoplus_{\om \in S \subset \cs_{+}(\ga)} \ch_{\om}$, respectively. We are requiring that the equality $V \left(\bigoplus_{\om \in S \subset \cs_{+}(\ga)} \G_{\om}\right)=\G V$ is satisfied.
Let $\xi=(\xi_{\om})_{\om}$ be a vector in $\bigoplus_{\om \in S \subset \cs_{+}(\ga)} \ch_{\om}$. It is clear that $V(\G_{\om}(\xi_{\om}))=\G(V(\xi_{\om}))$ on each component $\xi_{\om}$ if and only if $V$ sends even (odd) vectors to even (odd) ones.
Returning to the previous proposition, however, we note that an even vector can always be found. Indeed, if $\xi$ is even in $\ch_{\pi}$, the invariant space $M:=\overline{\pi(\ga)\xi}$ allows us to establish the equivalence $\pi_{\om} \simeq \pi_{\upharpoonright M}$, where the unitary given by the GNS construction sends even vectors in even ones.
If now $M \neq \ch_{\pi}$ and an even vector $\xi$ cannot be found in $M^{\perp}$, we can consider the vector $\pi(a)\xi \in M^{\perp}$, which is still even if $\xi$ and $\pi(a)$ are odd.
This ends the proof.
\end{proof}
Under the proposition above, it is then quite natural to define a $\bz_2$-graded analogue of the \emph{enveloping von Neumann algebra}. With this goal in mind,  we first introduce what we call the $\bz_2$-\emph{graded universal representation} as
$$
\pi_{\G}^{e}:=\bigoplus_{\om \in \cs_{+}(\ga)} \pi_{\om}
$$
By \cite[Proposition 2.1]{CRZ},  $\pi_{\G}^{e}$ is faithful. The \emph{$\bz_2$-graded enveloping von Neumann algebra} is $\mathcal{R}_{e}:=\pi_{\G}^{e}(\ga)^{''}$.\\
We will think of $\mathcal{R}_e$ as a $\bz_2$-graded algebra with the grading $\G_e:=\bigoplus_{\om \in \cs_{+}(\ga)} \G_{\omega}$, induced by the unitaries $\G_{\om}: \ch_{\om} \to \ch_{\om}$.\\
The von Neumann algebra $(\mathcal{R}_{e}, \G_e)$ enjoys the following universal property: for every grading-equivariant representation $\pi$ of $\ga$,
there exists a unique (grading-equivariant) epimorphism $\rho:\mathcal{R}_{e}\rightarrow \pi(\ga)''$ making the diagram below commutative
\[
\begin{tikzcd}[column sep=3.5em, row sep=3.5em]
\ga \arrow{r}{\pi_{\G}^{e}}  \arrow{rd}{\pi} 
  & \mathcal{R}_{e} \arrow{d} {\rho}\\
    & \pi(\ga)^{''}
\end{tikzcd}
\]
where $\pi(\ga)''$ is thought of as a $\bz_2$-graded algebra with respect to the adjoint action of  the unitary $\G$.\\
Clearly, $\rho$ is the unique extension to the algebra $\mathcal{R}_e$ of the map $\pi_{\G}^e(a) \to \pi(a)$,\, $a \in \ga$, which is well defined since $\pi_{\G}^e$ is faithful. 
In addition, $\rho$ is grading-equivariant. This follows by Proposition \ref{prop:enveloping} and the definition of $\G_e$ on $\mathcal{R}_e$.
\begin{ex}
\label{exa:enveloping}
Let $(\ga,\th)$ be $(\mathcal{K}(\ch),\g\lceil_{\mathcal{K}(\ch)})$, where $\mathcal{K}(\ch)$ denotes the compact operators on $\ch$. We denote by $\pi_0$ the $*$-representation of $\mathcal{K}(\ch)$ on $\ch$ given by the natural inclusion $\mathcal{K}(\ch) \subset \mathcal{B}(\ch)$. As one could expect, the $\bz_{2}$-graded enveloping von Neumann algebra $(\mathcal{R}_e,\G_e)$ is $*$-isomorphic with $(\cb(\ch),\g)$. This can be seen by showing that the pair $(\cb(\ch),\g)$ satisfies the universal property that uniquely determines $(\mathcal{R}_e,\G_e)$.
In order to show that the universal property holds, it is enough to remember that any $*$-representation of $\mathcal{K}(\ch)$ is unitarily equivalent
to a direct sum of (at most as many as the real numbers) representations each of which is  $\pi_0$.
\end{ex}

Finally, we recall some fundamental notions concerning the tensor product of $C^*$-algebras and the norms defined on it. \\
Consider the $C^*$-algebras $\ga_1$ and $\ga_2$, and denote by $\ga_1\otimes \ga_2$ the algebraic tensor product $\ga_1\odot \ga_2$ with the product and involution given by
$$
(a_1\otimes a_2)\cdot(a_1'\otimes a_2'):=a_1a_1'\otimes a_2a_2'\,,\quad (a_1\otimes a_2)^*:=a_1^*\otimes a_2^*\,,
$$
for all $a_1,a_1'\in\ga_1$, $a_2,a_2'\in\ga_2$. Let us denote by $\ga_1\otimes_{\max} \ga_2$ and $\ga_1\otimes_{\min} \ga_2$ the completion of $\ga_1\otimes \ga_2$ with respect to the maximal and minimal $C^*$-cross norm, respectively \cite[IV.4.]{T1}.\\
If one takes $\om_1\in\cs(\ga_1)$ and $\om_2\in\cs(\ga_2)$, their product state $\om_1 \otimes \om_2 \in\cs(\ga_1\otimes_{\min} \ga_2)$ is well defined also on $\ga_1\otimes_{\max} \ga_2$, and consequently the notation $\om_1 \otimes \om_2  \in\cs(\ga_1\otimes \ga_2)$ will be used in the sequel.\\
Suppose now that $(\ga_1,\th_1)$ and $(\ga_2,\th_2)$ are $\bz_2$-graded $*$-algebras, and consider the linear space $\ga_1\odot \ga_2$. In what follows, we recall the definition of the involutive $\bz_2$-graded tensor product, also called Fermi tensor product, which will be henceforth denoted by
$\ga_1\hat{\otimes} \ga_2$ (see \cite[Section 14.4]{Black}). Recall that, for $i=1,-1$, $\ga_{1,i}$ denotes the even or the odd part of $\ga_1$, respectively. Analogous notation can be applied to $\ga_{2,i}$. \\
For homogeneous elements $a_1\in\ga_1$, $a_2\in\ga_2$ and $i,j\in\bz_2$, we set
\begin{eqnarray*}
\eeps(i,j):=\left\{\!\!\!\begin{array}{ll}
                      -1 &\text{if}\,\, i=j=-1\,,\\
                     \,\,\,\,\,1 &\text{otherwise}\,.
                    \end{array}
                    \right.
\end{eqnarray*}
Given $x,y\in\ga_1\odot\ga_2$ with 
\begin{align*}
&x:=\oplus_{i,j\in\bz_2}x_{i,j}\in\oplus_{i,j\in\bz_2}(\ga_{1,i}\odot\ga_{2,j})\,,\\
&y:=\oplus_{i,j\in\bz_2}y_{i,j}\in\oplus_{i,j\in\bz_2}(\ga_{1,i}\odot\ga_{2,j})\,,
\end{align*}
the involution, which, by a minor abuse of notation, we continue
to denote by $*$, and  the multiplication on $\ga_1\hat{\otimes} \ga_2$ are defined as (see also \emph{e.g.} \cite{CDF}) 
\begin{align*}
x^*:=&\sum_{i,j\in\bz_2}\eeps(i,j)x_{i,j}^* & xy:=&\sum_{i,j,k,l\in\bz_2}\eeps(j,k)x_{i,j}{\bf\cdot}y_{k,l}\,.
\end{align*}
The $*$-algebra thus obtained also carries a $\bz_2$-grading, which
is induced by the $*$-automorphism $\hat{\th}=\th_1\, \hat{\otimes}\, \th_2$ given on the elementary tensors by
\begin{equation*}
(\th_1\,\hat{\otimes} \, \th_2)(a_1\, \hat{\otimes}\, a_2):=\th_1(a_1)\, \hat{\otimes}\, \th_2(a_2)\,,\quad a_1\in\ga_1\,,\,\, a_2\in\ga_2\,.
\end{equation*}
where $a_1\, \hat{\otimes}\, a_2$ is nothing but  $a_1\otimes a_2$ thought of
as an element of the $\bz_2$-graded $*$-algebra $\ga_1\, \hat{\otimes}\, \ga_2$,
since
$\ga_1\, \hat{\otimes}\, \ga_2=\ga_1\otimes \ga_2\,$ as linear spaces.\footnote{As of now, we will use $a_1\otimes a_2$ and $a_1\, \hat{\otimes}\, a_2$ interchangeably 
when no confusion can occur}. The even and odd part of the Fermi product are, respectively
\begin{equation*}
\begin{split}
\big(\ga_1\, \hat{\otimes} \ga_2\big)_+:=&\big(\ga_{1,+}\odot\ga_{2,+}\big)\oplus\big(\ga_{1,-}\odot\ga_{2,-}\big)\,,\\
\big(\ga_1\, \hat{\otimes} \, \ga_2\big)_-:=&\big(\ga_{1,+}\odot\ga_{2,-}\big)\oplus\big(\ga_{1,-}\odot\ga_{2,+}\big)\,.
\end{split}
\end{equation*}
For $\om_i\in\cs(\ga_i)$, $i=1,2$, the state $\om_1 \otimes \om_2$ has a counterpart in $\ga_1  \hat{\otimes} \ga_2$ by means of the product functional $\om_1 \times \om_2$, defined as usual by
$$
\om_1\times \om_2\bigg(\sum_{j=1}^n a_{1,j} \hat{\otimes} \, a_{2,j}\bigg):=\sum_{j=1}^n \om_1(a_{1,j})\om_2(a_{2,j})\,,
$$
for all $\sum_{j=1}^n a_{1,j} \hat{\otimes}\, a_{2,j}\in \ga_1\, \hat{\otimes}\, \ga_2$. Unlike the case of a trivial grading, the map defined above is not necessarily positive. However, it is positive as soon as at least one of the two states is even (see \cite[Proposition 2.6]{CRZ2}).\\
We also recall  that the spatial norm on the Fermi tensor product of $\bz_2$-graded $C^*$-algebras is defined in terms of the GNS representations of products
of even states: 
$$
\|x\|_{\min}:=\sup \{\|\pi_{\om_1\times\om_2}(x)\|: \om_1\in\cs_+(\ga_1)\,,\,\, \om_2\in\cs_+(\ga_2)\}\,,
$$
for all $x\in\ga_1\,\hat{\otimes}\,\ga_2$.
As shown in \cite[Theorem 4.12]{CRZ},  this norm is minimal besides being a cross norm, as is the maximal one introduced in \cite{CDF}. The latter is given  by
$$
\|x\|_{\max}:=\sup\{\|\pi(x)\|: \pi\,\,\text{is a representation}\}\,,
$$
for all $x\in\ga_1\, \hat{\otimes}\,\ga_2$, and it
is obviously the biggest norm on $\ga_1\, \hat{\otimes}\,\ga_2$.\\

\section{$C^*$-independence for $\bz_2$-graded $C^*$-subalgebras} 
\label{sec:C*independence}
In this section, after recalling some definitions and results concerning the $C^*$-independence for $C^*$-algebras with trivial grading, we provide an analogous definition for $\bz_2$-graded algebras, and we prove that the Schlieder condition  $(S)$ is a necessary condition for $C^*$-independence in this setting. Successively, we show that it is also a sufficient condition for $C^*$-independence of $C^*$-subalgebras commuting with grading. \\

We start by summarizing without proofs the relevant material on $C^*$-independence of subalgebras of a given $C^*$-algebra $\ga$ with trivial grading. \\
First, we recall that two subalgebras of a $C^*$-algebra $\ga$, $\ga_1$ and $\ga_2$,  are said to be \emph{$C^*$-independent}, $(C^*I)$, or \emph{statistically independent}, if taking a state $\f_1$ of $\ga_1$ and a state $\f_2$ of $\ga_2$, there exists a state $\f$ of $\ga$ whose restriction to $\ga_1$ equals $\f_1$ and the restriction to $\ga_2$ equals $\f_2$. Moreover, 
$\ga_1$ and $\ga_2$ satisfy the Schlieder property $(S)$ if $a_1a_2\neq 0$ for every non vanishing elements $a_1\in\ga_1$ and $a_2\in\ga_2$.  \\
We recall that in \cite{Sch}, the author proved that $(S)$ is a necessary condition for the $C^*$-independence of $C^*$-algebras. Moreover, if $\ga_1$ and $\ga_2$ commute elementwise, $(S)$ is also a sufficient condition for $(C^*I)$. We refer the reader to \cite{Roos} for a detailed proof.\\

Consider $(\ga,\th)$, a $\bz_2$-graded $C^*$-algebra with two $C^*$-subalgebras $(\ga_i,\th_i)$, $i=1,2$, with $\th_i:=\th\lceil\ga_i$, for each $i=1,2$. Suppose these two subalgebras share the identity of the $C^*$-algebra $\ga$, denoted by $\idd_{\ga}$. Assume also that $\th_i(\ga_i)\subset \ga_i$, $i=1,2$.\\
We proceed by giving our notion of $C^*$-independence in the $\bz_2$-graded setting. 
\begin{defn}
\label{def:Z2C*indep}
$C^*$-subalgebras $(\ga_1,\th_1)$ and $(\ga_2,\th_2)$ of the $\bz_2$-graded $C^*$-algebra $(\ga,\th)$ are $C^*$-independent, $(C^*I)_{\bz_2}$, if for any even states $\f_1\in\cs_+(\ga_1)$ and $\f_2\in\cs_+(\ga_2)$, there exists 
 a common even extension to $\ga$, that is, there is an even state $\f\in\cs_+(\ga)$ such that
 \begin{align*}
 \f\lceil\ga_1 &=\f_1 &  &\text{and} &\f\lceil\ga_2&=\f_2.
 \end{align*}
\end{defn}
The following result shows that the Schlieder property is necessary for $C^*$-independence.  
\begin{thm}
\label{thm:Z2Schlieder}
Let $(\ga_1,\th_1)$ and $(\ga_2,\th_2)$ two $\bz_2$-graded $C^*$-subalgebras of $(\ga,\th)$. Then $(C^*I)_{\bz_2} \Rightarrow (S)$. 
\end{thm}
\begin{proof}
We rephrase the proof of \cite[Theorem 2.5]{GoLuWi} as follows. By \cite[Proposition 2.1]{CRZ}, even states separates points of $\ga_i$, therefore there exist $\f_i\in\cs_+(\ga_i)$ with $\f_1(a_1)=\|a_1\|=1=\|a_2\|=\f_2(a_2)$, for non vanishing positive elements $a_i\in\ga_i$, $i=1,2$. 
By $(C^*I)_{\bz_2}$, there exists an even state $\f\in\cs_+(\ga)$ that extends $\f_1$ and $\f_2$. As a consequence, $\pi_{\G}^{e}(a_1a_2)y=y$, where $\pi_{\G}^{e}$ denotes the $\bz_2$-graded universal representation  and $y$ a unit vector in the Hilbert space $\ch_u:=\bigoplus_{\om\in\cs_+(\ga)}\ch_{\om}$, with $y:=\bigoplus_{\r\in\cs_+(\ga)} y_{\r}$, where $y_{\r}=0$ if $\r\neq \f$ and $y_{\r}= \xi_{\f}$ otherwise. 
This implies that $\pi_{\G}^{e}(a_1a_2)$ is not vanishing, giving thus $a_1a_2\neq0$ since $\pi_{\G}^{e}$ is faithful. \\
Now suppose that $a_1\in\ga_1$ and $a_2\in\ga_2$ are non zero elements (not necessarily positive). Then $a_i^{*}a_i\geq 0$, $i=1,2$, implies that $a_1^{*}a_1a_2a_2^{*}\neq 0$ by the first step of the proof. Therefore $a_1a_2\neq0$, which is the $(S)$ condition. 
\end{proof}
 The next result yields the converse of the previous one when the two $C^*$-subalgebras commute with grading. 
\begin{thm}
\label{thm:RoosZ2}
Let $(\ga_1,\th_1)$ and $(\ga_2,\th_2)$ two $\bz_2$-graded $C^*$-subalgebras of $(\ga,\th)$.  satisfying $(C)_{\bz_2}$. Then $(S) \Rightarrow (C^*I)_{\bz_2}$. 
\end{thm}
\begin{proof}
Let    $\f_1\in\cs_+(\ga_1)$ and $\f_2\in\cs_+(\ga_2)$ even states on $\ga_1$ and $\ga_2$ respectively. 
Denote by $\tilde{\f}_i:=\f_i\lceil\ga_{i,1}$, $i=1,2$ the restriction of $\f_i$ to the even subalgebra $\ga_{i,1}$. 
By assumption, the $(S)$ condition still holds for the even subalgebras $\ga_{1,1}$ and $\ga_{2,1}$. Therefore, since these two even subalgebras commute elementwise, by \cite[Theorem 1]{Roos}, there exists a state $\tilde{\f}\in\cs(\ga_{1,1}\vee\ga_{2,1})$ extending $\tilde{\f}_1$ and $\tilde{\f}_2$. Denote by $\bar{\f}$ its extension  to $\ga_+$ and define 
$$
\f:=\bar{\f} \circ \eps_1 \,.
$$
As $\f\lceil\ga_-$ is vanishing,  $\f$ is an even state on $\ga$ which extends both $\f_1$ and $\f_2$. 
\end{proof}

\section{$W^*$-Independence for $\bz_2$-graded von Neumann algebras}
\label{sec:W*independence}
In this section, we give several notions of independence for $\bz_2$-graded von Neumann algebras and prove some relationships among them. In particular, we adapt the definitions given in \cite{GoLuWi,Ham97} for von Neumann algebras with trivial grading to our setting. \\
Let $(\ca,\th)$ a $\bz_2$-graded von Neumann algebra of $\cb(\ch)$ and $(\ca_i,\th_i)$ von Neumann subalgebras of $(\ca,\th)$, with $\th_i:=\th\lceil_{\ca_i}$, $i=1,2$. The following definition is the analogue of $(C^*I)_{\bz_2}$ for $C^*$-algebras. 
\begin{defn}
\label{def:W*I}
$\bz_2$-graded subalgebras $(\ca_1,\th_1)$ and $(\ca_2,\th_2)$ are said to be \emph{$W^*I$-independent}, $(W^*I)_{\bz_2}$, if given even normal states $\om_1\in\cs_+(\ca_1)$ and $\om_2\in\cs_{+}(\ca_2)$, there exists an even normal state $\om\in\cs_+(\ca)$ s.t. $\om\lceil\ca_1 =\om_1$ and $\om\lceil\ca_2=\om_2$.
\end{defn}
\begin{defn}
\label{def:SL}
The ordered pair $(\ca_1,\ca_2)$ of $\bz_2$-graded von Neumann subalgebras of the $\bz_2$-graded von Neumann algebra $(\ca,\th)$ is said to satisfy \emph{strict locality}, $(SL)_{\bz_2}$, if for any nonzero projection $p_1\in\ca_1$ and any normal even state $\om_2\in\cs_+(\ca_2)$, there is a normal even state $\om\in\cs_+(\ca)$ s.t. $ \om(p_1) =1  $ and $\om\lceil\ca_2=\om_2$.
\end{defn}
Given nonzero projections $p_1\in\ca_1$ and $p_2\in\ca_2$, $p_1\wedge p_2$ denotes the projection on the close subspace $\ran(p_1)\cap\ran(p_2)$. Recall that $\bz_2$-graded von Neumann subalgebras $(\ca_1,\th_1)$ and $(\ca_2,\th_2)$ of the von Neumann algebra $(\ca,\th)$ are said to be \emph{logically independent}, $(LI)$, if $p_1\wedge p_2\neq0$ for any nonzero projection $p_1\in\ca_1$ and $p_2\in\ca_2$. In addition, \cite[Proposition 3.7]{GoLuWi} states that logically independent von Neumann algebras satisfy the $(Cross)$ condition, \emph{i.e.} $\|xy\|=\|x\|\|y\|$, for all $x\in\ca_1$, $y\in\ca_2$. 
Let $(\ga,\th)$ a $\bz_2$-graded $C^*$-algebra with unit $\id$. We notice that, by \cite[Proposition 4.3.3.]{Ka1} applied to an even self-adjoint subspace of $\ga$ and \cite[Lemma 2.3]{CRZ}, for any $A\in\ga_+$ there exists an even state $\r\in\cs_+(\ga)$ such that $\r(|A|)=\|A_+\|$.
The next proposition gives us the analogue of \cite[Proposition 3]{FloSum} for $\bz_2$-graded $C^*$-algebras, and its proof is very close to that given for the trivial setting. Therefore, we mention only the main steps. In what follows, $\s(A)$ denotes the spectrum of $A$. 
\begin{prop}
\label{prop:3FloSu}
Let $(\ga_1,\th_1)$ and $(\ga_2,\th_2)$ $C^*$-subalgebras of a $\bz_2$-graded $C^*$-algebra $(\ga,\th)$. Then $(Cross)\Rightarrow (C^*I)_{\bz_2}$.
\end{prop}
\begin{proof}
Let $a=a^*\in\ga_{1,+}$ and $b=b^*\in\ga_{2,+}$. By the considerations stated above, the property of $C^*$-norm and the Cauchy–Schwarz inequality, following the proof of \cite[Proposition 3]{FloSum}, there exists an even state $\f\in\cs_+(\ga)$, s.t. 
$\f(a)=\|a\|$ and $\f(b)=\|b\|$.\\
Consider, for fixed $b=b^*\in\ga_{2,+}$ and $\m\in[\min\s(b),\max\s(b)]$, the set of states
\begin{equation*}
  \cv:=\{\psi\in\cs_+(\ga_1)|\psi=\f\lceil_{\ga_1}\,, \text{for}\, \f\in\cs_+(\ga)\,\text{with}\, \f(b)=\m\}\,.
\end{equation*}
We will show that $\cv=\cs_+(\ga_1)$. Let $\{\f_n\}_n$, $\f_n\in\cs_+(\ga)$ a net of even states whose restrictions to $\ga_1$ are Cauchy in the weak$*$ topology and $\f_n(b)=\mu$, for any $n$. Since $\cs_{+}(\ga)$ is a weak$*$-compact convex set, $\{\f_n\}_n$ admits a convergent subnet, the limit of which, when restricted to $\ga_1$, is an element of $\cv$. This implies that $\cv$ is weak$*$ closed and convex. Suppose now that there exists an even state $\xi\in\cs_+(\ga_1)$, $\xi\notin \cv$.
By Hahn–Banach theorem, there exists a self-adjoint element $a\in\ga_{1,+}$ s.t. 
\begin{equation}
 \label{eq:xineqpsi} 
\xi(a)\neq \psi(a)\,, \quad\text{for all $\psi\in\cv$.} 
\end{equation}
Moreover, necessarily, $\xi(a)\in[\min\s(a),\max \s(a)]$. Therefore, by the first step of the proof, there exists an even state $\f\in\cs_+(\ga)$ s.t. $\f(a)=\xi(a)$ and $\f(b)=\mu$. Defining $\tilde{\psi}:=\f\lceil_{\ga_1}$, one has $\tilde{\psi}\in\cv$ and $\tilde{\psi}(a)=\f(a)=\xi(a)$, but this contradicts \eqref{eq:xineqpsi}. Hence, $\cv=\cs_+(\ga_1)$.\\
Finally, for a fixed $\f\in\cs_+(\ga_1)$, define 
\begin{equation*}
\cv^{\prime}:=\{\psi\in\cs_+(\ga_2)|\psi=\phi\lceil_{\ga_2}\,, \text{for}\, \phi\in\cs_+(\ga)\,\text{with}\, \phi\lceil_{\ga_1}=\f\}\,.
\end{equation*}
As above, $\cv^{\prime}$ is weak$*$-closed and convex and, if $\cs_+(\ga_2)\neq \cv^{\prime}$, there must exists a self-adjoint element $b\in\ga_{2,+}$ and an even state $\psi_0\in\cs_+(\ga_2)$ s.t.
\begin{equation}
 \label{eq:psi0neqpsi}
\psi_0(b)\neq \psi(b)\,,\quad \text{for all $\psi\in\cv^{\prime}$.} 
\end{equation}
Since $\psi_0(b)\in[\min\s(b),\max\s(b)]$, by the second step of the proof, there exists an even state $\chi\in\cs_+(\ga)$ with $\chi(b)=\psi_0(b)$ s.t. $\f=\chi\lceil_{\ga_1}$.  Denoting by $\upsilon:=\chi\lceil_{\ga_2}$, one has that $\upsilon\in\cv^{\prime}$ and $\upsilon(b)=\chi(b)=\psi_0(b)$, in contradiction with \eqref{eq:psi0neqpsi}. Therefore, $\cs_+(\ga_2)= \cv^{\prime}$ and this implies that $\ga_1$ and $\ga_2$ satisfy $C^*$-independence.   
\end{proof}
\begin{rem}
\label{rem:CrossCom}
If $(\ga_1,\th_1)$ and $(\ga_2,\th_2)$ are $C^*$-subalgebras of a $\bz_2$-graded $C^*$-algebra $(\ga,\th)\subseteq(\cb(\ch),\g)$ satisfying $(C)_{\bz_2}$, Proposition \ref{prop:3FloSu} is trivially satisfied. Indeed, given non vanishing elements $a\in\ga_1$ and $b\in\ga_2$, by hypothesis, one has $\|a\eta_{\G}(b)\|=\|a\|\cdot\|\eta_{\G}(b)\|\neq0$, 
and then $a\eta_{\G}(b) \neq0$, \emph{i.e.} the $(S)_{\bz_2}$ condition is satisfied. Therefore $(C^*I)_{\bz_2}$ follows by Theorem \ref{thm:RoosZ2}.  
\end{rem}
We recall that Proposition \ref{prop:3FloSu} still holds if $(\ca_1,\th_1)$ and $(\ca_2,\th_2)$ are von Neumann subalgebras of a $\bz_2$-graded von Neumann algebra, since the notion of $C^*$-independence is also applicable to a pair of
subalgebras of a $W^*$-algebra (see \cite{FloSum}). \\
\begin{thm}
Let $(\ca,\th)$ a $\bz_2$-graded von Neumann algebra on a Hilbert space $\ch$. Let $(\ca_1,\th_1)$ and $(\ca_2,\th_2)$ $W^*$-subalgebras of $(\ca,\th)$, with $\th_i:=\th\lceil_{\ca_i}$, $i=1,2$. Then
\begin{equation*}
  (W^*I)_{\bz_2}\Rightarrow (SL)_{\bz_2} \Rightarrow (LI) \Rightarrow (C^*I)_{\bz_2}
\end{equation*}
\end{thm}
\begin{proof}
$(W^*I)_{\bz_2}\Rightarrow (SL)_{\bz_2}$ Let $p_1$ a non zero projection $\ca_1$, $\xi\in\ch_+$ a unit even vector belonging to the range of $p_1$, ($p_1\xi=\xi$). The vector state determined by $\xi$, $\om_\xi$, determines a normal state on $\ca_1$. Therefore, $\om_1:=\om_{\xi}$ is a normal state on $\ca_1$ s.t. $\om_1(p_1)=1$. In addition, $\om_\xi$ is even by \cite[Lemma 3.1]{CRZ}.   Let $\om_2\in\cs_+(\ca_2)$ a normal even state on $\ca_2$. By $(W^*I)_{\bz_2}$, there exists a normal even state $\om\in\cs_+(\ca)$ which extends $\om_1$ and $\om_2$, with $\om(p_1)=\om_1(p_1)=1$.\\
$(SL)_{\bz_2} \Rightarrow (LI)$ Let $p_1\in\ca_1$ and $p_2\in\ca_2$ non zero projections. Let $\xi$ be a unit vector in the range of $p_2$ ($p_2\xi=\xi$). Denote by $\om_2:=\om_{\xi}$ the vector state determined by $\xi$. Then $\om_2$ is a normal even state on $\ca_2$, with $\om_2(p_2)=1$. By $(SL)_{\bz_2}$, there exists a normal even state $\om\in\cs_+(\ca)$ s.t. $\om(p_1)=1$ and $\om\lceil_{\ca_2}=\om_2$. Therefore, $\om(p_1)=1=\om(p_2)$, implies that the support $s(\om)$ of the normal state $\om$ satisfies $s(\om)\leq p_1$, and $s(\om)\leq p_2$. As a consequence, $s(\om)\leq p_1\wedge p_2$, and then $p_1\wedge p_2\neq 0$.   
  \item[$(LI) \Rightarrow (C^*I)_{\bz_2}$] It follows by \cite[Proposition 3.7]{GoLuWi} and Proposition \ref{prop:3FloSu}.  
\end{proof}

\section{Characterisation of the graded nuclearity property}
\label{sec:nuclearity}
This section aims to give a characterization for the nuclearity of $\bz_2$-graded $C^*$-algebras. We define nuclearity for a $\bz_2$-graded $C^*$-algebra and the normal tensor product.\\
Recall that a $C^*$-algebra $\ga$ is  \emph{nuclear} if, for every $C^*$-algebra $\gb$ the algebraic
tensor product  $\ga\otimes\gb$ can be endowed with only one $C^*$-norm.\\
When a $\bz_2$-grading is added, an analogue property can be defined as follows:
\begin{defn}
\label{def:nuclear}
A $\bz_2$-graded $C^*$-algebra $\ga$ is said to be $\bz_2$-\emph{nuclear} if the maximal and minimal C*-cross norms on $\ga \hat{\otimes} \gb$ are the same for every $\bz_2$-graded $C^*$-algebra $\gb$.
\end{defn}
Consider a $\bz_2$-graded von Neumann algebra $(\mathcal{R},\g\lceil_{\car})$ and a $\bz_2$-graded $C^*$-algebra $\gb$. Following \cite{EL}, we introduce the set
$$
\mathrm{nor}=\mathrm{nor}(\mathcal{R} \otimes \gb):= \{ \varphi \in \cs (\mathcal{R} \hat{\otimes} \gb) \,\, \mathrm{s.t.} \, T_{\varphi}(\gb) \subseteq \mathcal{R_{*}} \}
$$
where $T_{\varphi} \in \mathcal{B}(\gb, \mathcal{R_{*}})$ is defined by $T_{\varphi}(b)(r):=\varphi(r \otimes b)$, for any $r \in \mathcal{R}$ and $b \in \gb$.\\
The set $\mathrm{nor}(\mathcal{R} \otimes \gb)$ separates the elements of $\mathcal{R}\hat{\otimes} \gb$ since 
it contains all vector states associated with any vector in $\ch \otimes \ck$, $\ch$ and $\ck$ being the Hilbert spaces on which $\mathcal{R}$ and $\gb$ act, respectively.
\begin{defn}
\label{def:normal}
For any $x \in \mathcal{R} \hat{\otimes} \gb$, the completion w.r.t. the norm 
$$
\|x\|_{\mathrm{nor}}= \sup \{\varphi(xx^{*})^{\frac{1} {2}}: \varphi \in \mathrm{nor}\}
$$
is the \emph{normal tensor product} $\mathcal{R} \otimes_{\mathrm{nor}} \gb$.
\end{defn}
Also recall that a norm on $\ga \hat{\otimes} \gb$ is said to be \emph{compatible} if the natural grading $\th_1 \hat{\otimes} \th_2$ extends to a $*$-automorphism of the completion (see \cite{CRZ} for further details).
\begin{lem}
\label{lem:nor}
$\| \cdot  \|_{\emph{nor}}$ is compatible.
\end{lem}
\begin{proof}
Let $x$ be an element on $\mathcal{R} \hat{\otimes} \gb$. Then
\begin{align*}
\| \th(x)\|_{\mathrm{nor}}&=\sup \{ \psi(\th(x)^*\th(x))^{\frac{1}{2}}: \psi \in \mathrm{nor}(\mathcal{R}\hat{\otimes} \gb)\}\\
&=\sup \{ \psi(\th(x^*)\th(x))^{\frac{1}{2}}: \psi \in \mathrm{nor}(\mathcal{R}\hat{\otimes} \gb)\}\\
&=\sup \{\psi(\th(x^*x))^{\frac{1}{2}}: \psi \in \mathrm{nor} (\mathcal{R} \hat{\otimes} \gb)\}\\
&=\sup \{\psi' (x^*x)^{\frac{1}{2}}: \psi'\in \mathrm{nor} (\mathcal{R} \hat{\otimes} \gb)\}\\
&=\| x \|_{\mathrm{nor}}\,,
\end{align*}
since any ultraweakly continous functional $\psi'$ can be written as $\psi \circ \th$ (for the adjoint map $r \to \G r\G^*$, $\forall \,r \in \mathcal{R}$ is ultraweakly continous).
\end{proof}
The completion w.r.t. the $\emph{normal}$ norm turns out to be crucial in giving a characterization of graded nuclearity. Indeed:
\begin{lem}
\label{lem:vectorstate}
Let $\ga$ be a $\bz_2$-graded $C^*$-algebra. An even state on $\ga$ can be seen as a vector state on $\mathcal{R}_{e}$.
\end{lem}
\begin{proof}
Let $\varphi \in \cs_+(\ga)$.  Then, in light of Proposition \ref{prop:cyclic} there exists an even unitary vector $x_{\varphi} \in \ch_{\G}=\bigoplus_{\varphi \in \cs_+(\ga)} \ch_{\varphi}$ defined as $x_{\varphi}(\varphi')=\xi_{\varphi} \delta_{\varphi, \varphi'}$, where $\xi_{\varphi} \in \ch_{\varphi},$\, such that $\forall a \in \ga$
\begin{align*}
\varphi(a)=\langle \pi_{\varphi}(a) \xi_{\varphi}, \xi_{\varphi}\rangle=\biggl \langle \bigoplus_{\varphi' \in \cs_{+}(\ga)}\pi_{\varphi'}(a) x_{\varphi},x_{\varphi} \biggl \rangle=
\langle \pi_{\G}^{e}(a)x_{\varphi},x_{\varphi}\rangle,
\end{align*}
with $\pi_{\varphi}$ grading-equivariant representation.
\end{proof}
\begin{thm}
\label{thm:noralg}
Let $\ga$ be a $\mathbb{Z}_2$-graded $C^*$-algebra. Then the following are equivalent:
\begin{itemize}
\item[i)] For any $\bz_2$-graded $C^*$-algebra $\gb$,
$$
\gb \hat{\otimes} _{\max} \ga=\gb \hat{\otimes}_{\min} \ga
$$
\item[ii)] For any $\bz_2$-graded von Neumann algebra $\mathcal{R}$,
$$
\mathcal{R} \hat{\otimes}_{\emph{nor}} \ga=\mathcal{R} \hat{\otimes}_{\min}\ga
$$
\end{itemize}
\end{thm}
\begin{proof}
i) $\Rightarrow$ ii) It follows from the assumption that $\mathcal{R} \hat{\otimes}_{\max} \ga=\mathcal{R} \hat{\otimes}_{\min}\ga$
for any $\bz_2$-graded von Neumann algebra $\mathcal{R}$, \emph{i.e.} $\|\cdot\|_{\min}=\|\cdot\|_{\max}$. \\
By virtue of Theorem 4.12 in \cite{CRZ}, we get the inequality
$\| \cdot \|_{\min} \leq \| \cdot \|_{\mathrm{nor}}$. 
On the other hand, the inquality $\|\cdot\|_{\mathrm{nor}}\leq \|\cdot\|_{\max}$ holds as well.
Putting the two inequalities together, we find
$$ \| \cdot \|_{\min} \leq \| \cdot \|_{\mathrm{nor}} \leq \| \cdot \|_{\max} \leq \| \cdot \|_{\min}\,,
$$
which proves the equality stated in ii).\\
ii) $\Rightarrow$ i)
We will denote by $\cs_{+}^{\max}(\gb \hat{\otimes} \ga)$ and $\cs_{+}^{\min}(\gb \hat{\otimes} \ga)$ the set of those even states that extend to the completion w.r.t. the maximal and the minimal norm, respectively. Since by Lemma \ref{lem:vectorstate} any $\varphi$ in $\cs_{+}^{\max}(\gb \hat{\otimes} \ga)$ can be viewed as a vector state on $\mathcal{R}_e$,  the equality $\cs_{+}^{\max}(\gb \hat{\otimes} \ga)=\mathrm{nor}(\mathcal{R}_e \hat{\otimes} \ga)$ holds.
Accordingly, there exists a net of weakly continuous states $\{\varphi_{\nu}\}$ on $\mathcal{R}_e \hat{\otimes} \ga$ which converges $*$-weakly to the state $\varphi$.
Furthermore, every weakly continuous state $\varphi_{\nu}$ is a finite convex combination of vector states $\langle \, \cdot \, \xi, \xi \rangle$, where $\xi$ is an even unitary vector on $\ck \hat{\otimes} \ch$, $\ck$ and $\ch$ being the Hilbert spaces on which $\mathcal{R}_e$ and $\ga$ act respectively. Then, for every index $\nu$, the map $\phi_{\nu}$ defined through the linear isomorphism $\varphi_{\nu} \to \phi_{\nu}$ by $\phi_{\nu} (a)(r):=\varphi_{\nu}(r \otimes a)$, is a finite rank map from $\ga$ into ${\mathcal{R}_e}_*$ ($\varphi_{\nu}$ being $*$-weakly continuous w.r.t. their first variable). It is apparent that, for every $\nu$ and a given $a \in \ga$, the functional $\phi_{\nu}(a)=\varphi(\cdot \,\hat{\otimes} a)$ is continuous w.r.t. $\| \cdot \|_{\min}$.  Since $\phi_{\nu }(a) \xrightarrow{w*} \phi(a)$ we have that $\forall a \in \ga, \,\phi(a)$ is continuous too w.r.t. the minimal norm, \emph{i.e.} $\varphi$ is in $\cs_+^{\min}(\gb \hat{\otimes} \ga)$.
This ends the proof thanks to the equality $\cs_+^{\max}(\gb \hat{\otimes} \ga)=\cs_+^{\min}(\gb \hat{\otimes} \ga)$.
\end{proof}

\subsection*{Acknowledgments}
The first named author is supported by Italian PNRR MUR project PE0000023-NQSTI, CUP H93C22000670006 and by Progetto ERC SEEDS UNIBA ``$C^*$-algebras and von Neumann algebras in Quantum Probability", CUP H93C23000710001.

\end{document}